%% file: Conway_infinite_order.tex
\documentclass[11pt,oneside]{amsart}
\usepackage[margin=1in]{geometry}
\usepackage{graphicx,color}
\usepackage{hyperref}
\usepackage{amssymb,cleveref}
\usepackage{tikz}
\usepackage[normalem]{ulem} 
\usepackage{thmtools, thm-restate}
\usepackage{float}
\usepackage{pinlabel}

\definecolor{azul}{HTML}{0072B2}
\definecolor{verde}{HTML}{009E73}
\definecolor{rojo}{HTML}{D55E00}

\newcommand{\C}{{\ensuremath{\mathbb{C}}}}

\newcommand{\Z}{{\ensuremath{\mathbb{Z}}}}
\newcommand{\Q}{{\ensuremath{\mathbb{Q}}}}
\newcommand{\F}{{\ensuremath{\mathbb{F}}}}
\newcommand{\T}{{\ensuremath{\mathcal{T}}}}
\newcommand{\Con}{\mathcal{C}}

\DeclareMathOperator{\HF}{HF^+}
\DeclareMathOperator{\CF}{CF^+}
\DeclareMathOperator{\CFK}{\widehat{CFK}}

\DeclareMathOperator{\lk}{lk}

\theoremstyle{plain}
\newtheorem{theorem}{Theorem}[section]

\newtheorem{proposition}[theorem]{Proposition}

\newtheorem{conj}[theorem]{Conjecture}

\theoremstyle{definition}

\theoremstyle{remark}
\newtheorem{remark}[theorem]{Remark}

\title{The Conway knot has infinite order in smooth concordance}
\author{Marco Golla}
\address{CNRS and Nantes Universit\'e, Nantes, France}

\author{Juanita Pinz\'on-Caicedo}
\address{University of Notre Dame, Notre Dame, IN USA}

\begin{document}
\maketitle
\begin{abstract}
We prove that the Conway knot has infinite order in the concordance group, and in particular that it is not slice. Ours is an alternative, independent proof with respect to those of Piccirillo and of Donatone, Kegel, Lewark, and Tru\"ol.
\end{abstract}

\section{Introduction}

Two knots $K_0$ and $K_1$ are said to be smoothly concordant if there is a smooth embedding of an annulus $[0,1]\times S^1$ into $[0,1]\times S^3$ that restricts to the given knots at each end. Requiring such an embedding to be locally flat instead of smooth gives rise to the notion of topological concordance. Both kinds of concordance are equivalence relations, and the sets of smooth and topological concordance classes of knots are denoted by $\Con$ and $\Con_{\rm TOP}$, respectively. Moreover, both $\Con$ and $\Con_{\rm TOP}$ are abelian groups with connected sum as their binary operation, identity element the class of knots that bound embedded disks (also called slice knots), and inverses given by reversing orientations and taking mirror images. The Conway knot $C$ (included to the left of \Cref{unknotting}) is an 11-crossing knot that is trivial in $\Con_{\rm TOP}$, but only proven to be non-trivial in $\Con$ by Piccirillo \cite{Piccirillo-Conway} after many decades of unknown status. In this note not only do we provide a new proof of Piccirillo's result, but do so using a more direct, and arguably natural, route. In fact, our chosen path takes us immediately to the following much stronger statement:

\begin{theorem}\label{order-C}
The Conway knot $C$ has infinite order in $\Con$.
\end{theorem}

We do so by showing that the 4-fold cover of $S^3$ branched over $C$ is of infinite order in the homology cobordism group.

\begin{theorem}\label{order-cover}
The integer homology sphere $\Sigma_4(C)$ has infinite order in $\Theta^3_\Q$.
\end{theorem}

Let us see how \Cref{order-cover} implies \Cref{order-C}.
This is well-known to experts, but we include a sketch of it for completeness.

\begin{proof}[Proof of \Cref{order-C}]
For any prime power $q$, the map $K \to \Sigma_q(K)$ induces a homomorphism $\Con \to \Theta^3_\Q$. This is because the $q$-fold cover of $S^3$ branched over a knot $K$ is a rational homology sphere, and the $q$-fold cover of $[0,1]\times S^3$ branched along a concordance from $K$ to $K'$ is a rational homology cobordism from $\Sigma_q(K)$ to $\Sigma_q(K')$. Thus, if $\Sigma_4(C)$ has infinite order in $\Theta^3_\Q$, a fortiori $C$ has infinite order in $\Con$.
\end{proof}

The proof of \Cref{order-cover} has two steps. The first is topological, and consists in giving a presentation of $\Sigma_4(C)$ as a surgery on a two-component link. The second is Floer-theoretic, and it consists in showing that the Heegaard Floer correction term of $\Sigma_4(C)$ does not vanish. The first step is rather classical, if a bit laborious. The second step, which ultimately relies on a relatively small computer calculation, makes use of a very minimal knowledge of Heegaard Floer homology and knot Floer homology. We note that Theorem~\ref{order-C} was recently independently proven by Donatone, Kegel, Lewark, and Tru\"ol~\cite{DKLT}.

\begin{figure}[h]
\centering
\def\svgwidth{0.275\textwidth}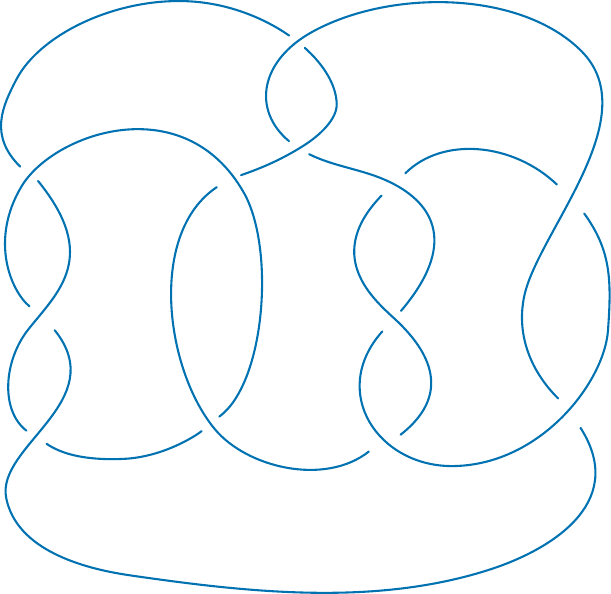 \qquad
\def\svgwidth{0.275\textwidth}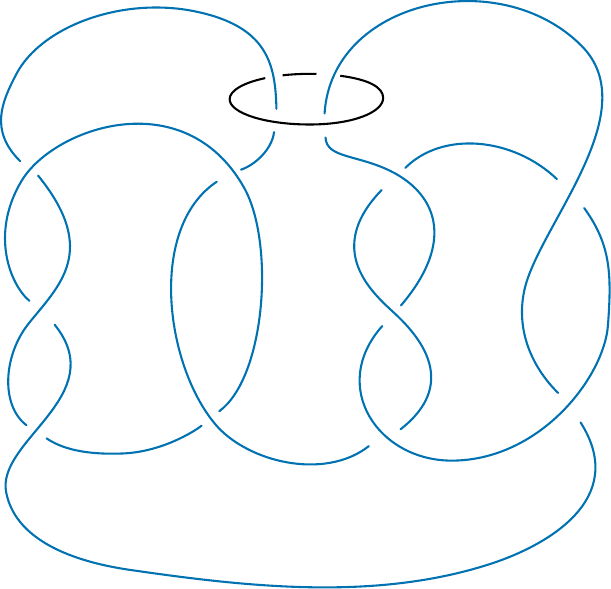 \qquad
\def\svgwidth{0.25\textwidth}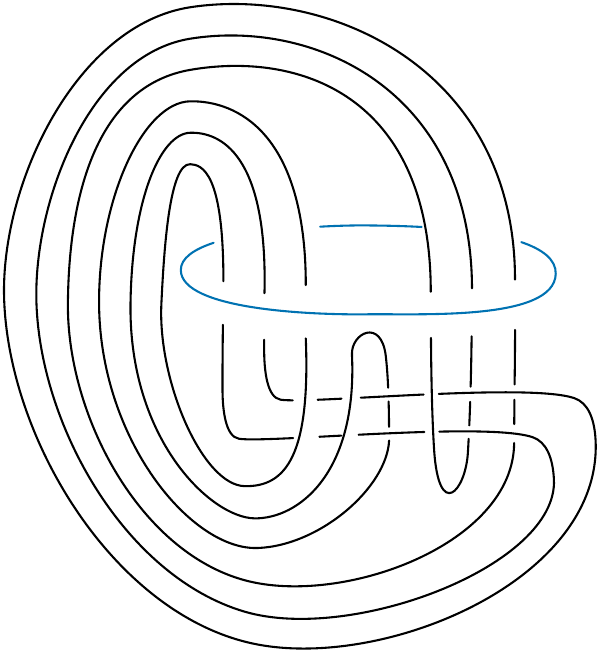
\caption{Left: The Conway knot. Middle and right: Unknotting sequence for the Conway knot. The curve $C'$ is unknotted in $S^3$, and it becomes $C$ in $S^3=S^3_{-1}(\gamma)$}\label{unknotting}
\end{figure}

\subsection{Motivation and perspectives}
There are multiple reasons why we had to wait so long to prove that $C$ is not slice:
on the one hand the knot $C$ has Alexander polynomial $\Delta_C(t)=1$, and work of Freedman implies that $C$ is topologically slice~\cite{Freedman-top, DET}.
Thus, no classical invariant captures information about its smooth concordance type.
On the other hand, $C$ is a mutant of the Kinoshita--Terasaka knot $KT$ \cite{KT}, and the latter is ribbon, and hence smoothly slice\footnote{Conway is credited with the construction of $C$ as a mutant of $KT$ and the reference given is \cite{conway}. We note that this reference does not highlight the knot $C$, nor does it compute its Alexander polynomial. A more accurate reference seems to be the actual talk given by John Conway at the 1967 conference \emph{Computational problems in abstract algebra}, whose Proceedings contain \cite{conway}.}.
As a consequence, to obtain information about the smooth concordance class of $C$, it is necessary to use smooth invariants (thus excluding metabelian obstructions) that in addition are not invariant under mutation (ruling out, for instance, the Jones, HOMFLY, and Kauffman polynomials, the S-equivalence class, and 2-fold branched covers). \\

Piccirillo's proof in~\cite{Piccirillo-Conway} uses the notion of RGB-links, and makes use of the trace embedding lemma and of Rasmussen's $s$-invariant in Khovanov homology. It does, however, fall short of proving that $C$ has infinite order in $\Con$. Donatone, Kegel, Lewark, and Tru\"ol's proof in~\cite{DKLT} also makes use of the $s$-invariant in Khovanov homology, but this time in conjuction with delicate constructions of satellite operations. Both of these results, like ours, rely on computer-aided calculations. Nevertheless, our proof is arguably more direct and follows a well-treaded path (see, for instance, Manolescu and Owens's beautiful paper~\cite{ManolescuOwens} and its generalisation by Jabuka~\cite{Jabuka}), and it has the further advantage of showing that $C$ is not infinitely divisible in $\Con$. It does not, however, show that $C$ generates a $\Z$-summand (nor that it belongs to a $\Z$-summand), which would answer~\cite[Question~5]{DKLT}; see \Cref{r:summand} below.\\

In fact, our proof is slightly more flexible, and we can prove that some of the \emph{generalised Conway knots} $C_{r,n}$ (shown in \Cref{Cmn}) have infinite order. These are the knots shown in \Cref{Cmn}; note that $C = C_{2,1}$.

\begin{theorem}\label{C2n}
The knot $C_{2,n}$ has infinite order in $\Con$ for each $n \neq 0$.
\end{theorem}

Note that, just like $C$, each of these knots is a mutant of a symmetric union of a diagram for the unknot. In particular, many known concordance invariants will fail to detect its non-sliceness. However, the $d$-invariant of the 4-fold cover does.

We were not able to push the techniques to show that every generalised Conway knot $C_{r,n}$ has infinite order, but we do make the following conjecture.

\begin{conj}
The generalised Conway knot $C_{r,n}$ has infinite order in $\Con$ for $r, n \neq 0$.
\end{conj}

In fact, we expect that for each $r$ there exists a $k$ such that $\Sigma_{2^k}(C_{r,n})$ has non-zero Heegaard Floer correction term.
We also note that, if the minimal such $k$ grows larger with $|r+2n|$, this automatically show that the set $\{C_{r,n}\}$ generates a $\Z^\infty$-subgroup of $\Con$.

\begin{figure}[h]
\centering
\def\svgwidth{0.375\textwidth}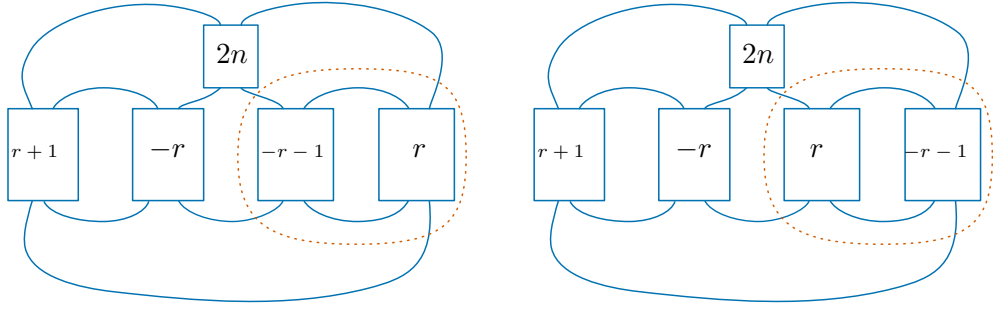\qquad
\def\svgwidth{0.375\textwidth}\input{KT_r,n.pdf_tex}
\caption{Left: The family of generalised Conway knots $C_{r,n}$. Right: The family of generalised Kinoshita--Terasaka knots $KT_{r,n}$. For every fixed tuple $(r,2n)$ the two knots $C_{r,n}$, $KT_{r,n}$ are mutants of one another, that is, one can be obtained from the other after cutting out the region inside the dashed line, and replacing it by its image under a rotation of $180^\circ$. The labels inside the rectangular regions represent the (signed) number of half-twists.}\label{Cmn}
\end{figure}

\subsection*{Acknowledgements}
We thank Tye Lidman for enriching conversations, and Marc Kegel, Lisa Piccirillo, and Paula Tru\"ol for comments on an earlier draft.

\subsection*{Computational Resource Disclosure} The authors did not use AI or LLM tools for any aspect of this research or the writing of this manuscript. SnapPy~\cite{SnapPy} and KLO~\cite{KLO} were used for all computations in support of this project.

\section{Surgery description for $\Sigma_4(C)$}\label{surgery}

In this section we introduce $\Sigma_4(C)$, the 4-fold cover of the 3-sphere branched along the Conway knot, and describe a very convenient surgery presentation for it. This is the first step in our proof of \Cref{order-cover}.

\begin{proposition}\label{p:two-component}
Let $C$ be the Conway knot and denote by $\Sigma_4(C)$ the $4$-fold cyclic cover of $S^3$ branched over $C$. The manifold $\Sigma_4(C)$ admits a surgery description along a $2$-component link $J_0\cup J_1$ with linking matrix $\left[\begin{smallmatrix}-1 & \phantom{-}0 \\\phantom{-}0 & -1\end{smallmatrix}\right]$.
\end{proposition}

\begin{figure}[h]
\centering
\def\svgwidth{0.4\textwidth}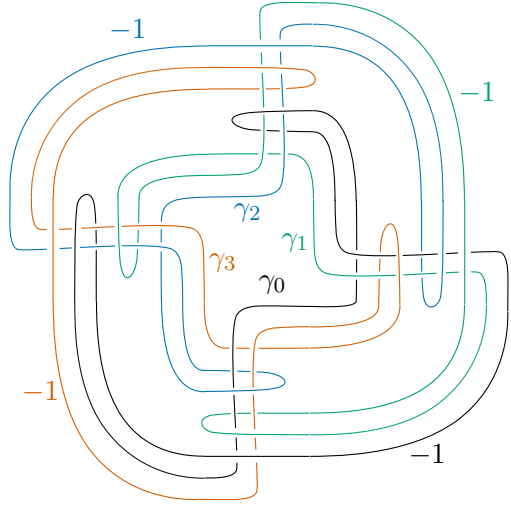\qquad
\caption{Surgery description of $\Sigma_4(C)$ obtained by lifting an unknotting curve $\gamma\subseteq S^3\setminus N(C)$.}\label{4-fold}
\end{figure}

\begin{proof} A surgery description of the branched cover $\Sigma=\Sigma_4(C)$ can be obtained from the surgery description of $C$ shown in \Cref{unknotting}. More precisely, let $C'\cup \gamma$ be the 2-component link that appears in the middle image of \Cref{unknotting}. Notice that $\lk(C',\gamma)=0$, and each of $C' ,\gamma$ is unknotted in $S^3$. Then $C$ is the knot obtained from $C'$ after performing $-1$-surgery on $S^3$ along $\gamma$.

Consider the covering map $f:\Sigma_4(C')\to S^3$. Since $C'$ is unknotted, then $S^3=\Sigma_4(C' )$. Moreover, since $\lk(C',\gamma)=0$, the link $f^{-1}(\gamma)$ consists of four components $\gamma_0\cup \gamma_1\cup \gamma_2\cup \gamma_3$. This is shown in \Cref{4-fold} and it is easy to see there that each curve $\gamma_i$ is unknotted.

Let us look at $\gamma$ as a framed knot, with framing $-1$ with respect to the Seifert framing.
Since $\lk(C',\gamma) = 0$, the lift of $\gamma$ as a framed link gives a surgery description of $\Sigma_4(C)$.
Let $N(\gamma) \subset S^3\setminus C'$ be a small tubular neighbourhood of $\gamma$, and $N_i \subset S^3$ be the component of $f^{-1}(N(\gamma))$ that contains $\gamma_i$, for $i = 0,\dots,3$. To compute the surgery coefficients for this description, let $\gamma^*$ be the $(-1,1)$-curve in $\partial N(\gamma)$, and denote by $\gamma^*_i$ the lift of $\gamma^*$ that lies in $\partial N_i$. The manifold $\Sigma$ is then the filling of
\[
f^{-1}\left(S^3\setminus N(\gamma)\right)=\Sigma_4(C' )\setminus\left(N_0\cup N_1\cup N_2\cup N_3 \right)
\]
that caps off the curves $\gamma^*_i$ with the meridional disks of the attached solid torus.
Thus, the framing numbers are precisely given by $\lk(\gamma_i,\gamma^*_i).$
Denote by $\mu,\lambda$ a meridional-longitude pair for $\gamma\subseteq S^3$, and by $\mu_i,\lambda_i$ their respective lifts to $\partial N_i$ so that the framing curve $\gamma^*_i$ is homologous to $-\mu_i+\lambda_i$. Notice that the curve $\mu_i$ is nothing other than a meridional curve for $\gamma_i$ in $S^3=\Sigma_4(C')$. To better understand the curve $\lambda_i$, denote by $D$ a disc in $S^3$ bounded by $\gamma$, which we can assume to be transverse to $C'$. The surface $F=f^{-1}(D)$ is then a surface with boundary $\gamma_0\cup \gamma_1\cup \gamma_2\cup \gamma_3$, and $\lambda_i$ is given by pushing $\gamma_i$ into the interior of $F$. This shows that $\lambda_i$ is homologous to the link $-\bigcup_{j \neq i} \gamma_j$ and so satisfies
\[
\lk(\gamma_i,  \lambda_i)= \lk(\gamma_i, -\cup_{j \neq i} \gamma_j)=-\sum_{j \neq i} \text{lk}(\gamma_i, \gamma_j)=0.
\]
The last equality follows from the explicit computation $\lk(\gamma_i,\gamma_j)=0$ $(i\neq j)$ using the diagram in \Cref{4-fold}. This shows that the surgery coefficients are given by
\[
\lk(\gamma_i,\gamma^*_i)=\lk(\gamma_i,-\mu_i+\lambda_i)=-\lk(\gamma_i,\mu_i)+\lk(\gamma_i,\lambda_i)=-1.
\]

Next we want to reduce the number of components in the surgery description for $\Sigma$. A series of elementary isotopies show that the link $\gamma_2\cup\gamma_3$ is the 2-component unlink (see \Cref{consecutive}). Since each component has framing number $-1$, they can both be blown down. The knot $J_i$ is the result of the blow-down on the knot $\gamma_i$ (indexes taken mod $4$). Since $\lk(\gamma_i,\gamma_j)=0$, for $i\in\{0,1\},\, j\in\{2,3\}$, the framing number of $J_i$ is also $-1$, and $\lk(J_0,J_1)=\lk(\gamma_0,\gamma_1)=0$, as sought.
\end{proof}

\begin{remark}
An alternative way of computing the surgery framing along $\gamma_i$ in $S^3$ is the following (see~\cite[Section 10.D]{Rolfsen}). The blackboard framing for a given diagram lifts to the blackboard framing when taking covers as in the previous proof, and the blackboard framing differs from the Seifert framing by the writhe of the diagram. Now, in the proof above the writhes of the diagrams of $\gamma$ and of $\gamma_i$ are both zero, and therefore the $-1$-framing with respect to the Seifert framing is also the $-1$-framing with respect to the blackboard framing. It follows that the $-1$-framing lifts to the $-1$-framing.
\end{remark}

\section{Heegaard Floer homology}\label{HF}

Heegaard Floer homology is a package of invariants of closed, oriented, and connected 3-manifolds, which was introduced by Ozsv\'ath and Szab\'o in~\cite{OSz-HD3MI, OSz-PA}.
Since the only 3-manifolds that we will be concerned with in this paper are integer homology spheres, we give a description of the invariant which is tailored for this context.
We will work with coefficients in $\F = \Z/2\Z$, the field of two elements.

The flavour of Heegaard Floer homology we are interested in is the `plus' version.
To each integer homology 3-sphere $Y$, it associates a $\Z$-graded $\F[U]$-module $\HF(Y)$. Multiplication by $U$ is an endomorphism of $\HF(Y)$ of degree $-2$.
$\HF(Y)$ is the homology of a complex $\CF(Y)$, which depends on the choice of a Heegaard diagram for $Y$ and of a basepoint (which we suppress from the notation).
$\HF(Y)$ contains a unique $\F[U]$-submodule which is isomorphic to $\T^+ := \F[U,U^{-1}]/U\cdot \F[U]$.
The minimal degree of an element in this distinguished submodule is called the correction term of $Y$, and it is denoted with $d(Y)$.

We will now list the properties of correction terms that we need.

\begin{theorem}[{Ozsv\'ath--Szab\'o~\cite{OSz-absolutely}}]\label{t:dinv}
Let $Y$ and $Y'$ be (oriented) integer homology spheres.
\begin{itemize}
\item $d(Y)$ is an even integer.
\item $d(Y\# Y') = d(Y) + d(Y')$.
\item If there exists a negative definite cobordism from $Y'$ to $Y$, then $d(Y') \le d(Y)$. In particular, if $Y$ bounds a rational homology ball, then $d(Y) = 0$.
\end{itemize}
\end{theorem}

There is a relative version of Heegaard Floer homology for knots in 3-manifolds~\cite{Rasmussen-PhD, OSz-knots, OSz-tau}. We restrict to knots in $S^3$, which is enough for our purposes.

If $K$ is an oriented knot in $S^3$, we can find a complex $\CF(S^3,K)$ coming from a Heegaard diagram of $S^3$ which is adapted to the knot (in a suitable sense). The complex $\CFK(S^3,K)$, defined as the cokernel of multiplication by $U$ on $\CF(S^3,K)$, is a chain complex whose homology is 1-dimensional. It comes equipped with an Alexander filtration $A$ which is compatible with the differential. The $\tau$-invariant of $-K$, the mirror of $K$, is the minimal Alexander filtration level such that the inclusion induces a surjection on homology: more formally, the inclusion $\CFK_{A \le s}(S^3,K) \hookrightarrow \CFK(S^3,K)$ induces the zero-map for $s < \tau(-K)$, and a surjection for $s \ge \tau(-K)$. Moreover, $\tau(-K) = -\tau(K)$.

\begin{theorem}[{\cite[Prop. 1.6]{NiWu}},{\cite[Prop 2.3]{hom-wu}}]\label{t:homwu}
If $\tau(K) < 0$, then $d(S^3_{-1}(K)) > 0$.
\end{theorem}

\section{Infinite order}
In this section, we use the topological input from \Cref{surgery} and the Floer-theoretic invariants from \Cref{HF} to prove \Cref{order-cover}. The key result is the following. 

\begin{proposition}\label{p:tau-negative} Let $J_0\cup J_1$ be as in \Cref{p:two-component}, so that $\Sigma_4(C)$ is obtained as $(-1,-1)$-surgery on $J_0 \cup J_1$. Then we have 
$\tau(J_0)=-1$.
\end{proposition}

\begin{proof}
The computation was performed in SnapPy~\cite{SnapPy}, by importing a knot diagram from KLO~\cite{KLO}. The KLO and SnapPy files used in the computation are available at \cite{github}.
\end{proof}

\begin{remark}\label{tau-bounds} While the precise calculation $\tau(J_0)=-1$ required the use of Szab\'o's algorithm, simple bounds for this quantity can be obtained `by hand'. Indeed, recall that $J_0$ is obtained from $\gamma_0$ after blowing down $\gamma_2\cup\gamma_3$. Since the link $\gamma_2 \cup \gamma_3$ is a $-1$-framed unlink and $\lk(\gamma_0,\gamma_2) = \lk(\gamma_0,\gamma_3) =  0$, by~\cite[Theorem~5.17]{HeddenRaoux} $\tau(J_0) \le 0$. A lower bound can be obtained from a convenient description of the knot $J_0$ as a banding of a 3-component link along two bands. First, a series of isotopies (included in \Cref{non-consecutive}) show that although the sublink $\gamma_0\cup\gamma_2$ is not the unlink, it is nevertheless a slice link. This gives a description of $\gamma_0$ as the banding of a $2$-component unlink $\alpha'_1\cup \alpha'_2$ along a band $b_1$. See the left image of \Cref{L023}. Next, incorporate $\gamma_3$ into the picture and compute $\lk(\alpha'_i,\gamma_3)=1$, with $\alpha'_1$ a meridional curve of $\gamma_3$. The curve $\alpha'_2$ is not quite a meridian for $\gamma_3$, but it is the banding of a meridian $\alpha_2$, and a curve $\alpha_0$ unlinked from both $\gamma_2$ and $\gamma_3$ (although not simultaneously). See the right image of \Cref{L023}. The curve $\gamma_0$ can thus be described as the banding of an unlink $\alpha_0\cup\alpha_1\cup\alpha_2$ along two bands $b_1,b_2$. Blowing down $\gamma_2\cup\gamma_3$ transforms $\alpha_0$ into $T_{2,-3}$, and $\alpha_1\cup\alpha_2$ into a Hopf link as in the middle of \Cref{J0}. Attaching a 0-framed band joining the two components of the Hopf link (say parallel to the curve $\gamma_3$ in the left diagram in \Cref{J0}) gives a planar cobordism from $J_0$ to another Hopf link with oppositely oriented components, thus showing that $g_4(J_0)\leq 1$. Since $|\tau|$ is a lower bound for $g_4$, we have that $-1 \le \tau(J_0) \le 0$. A posteriori, since $\tau(J_0) = -1$, we deduce that $g_4(J_0) = 1$.
\end{remark}

\begin{figure}
\centering
\def\svgwidth{0.4\textwidth}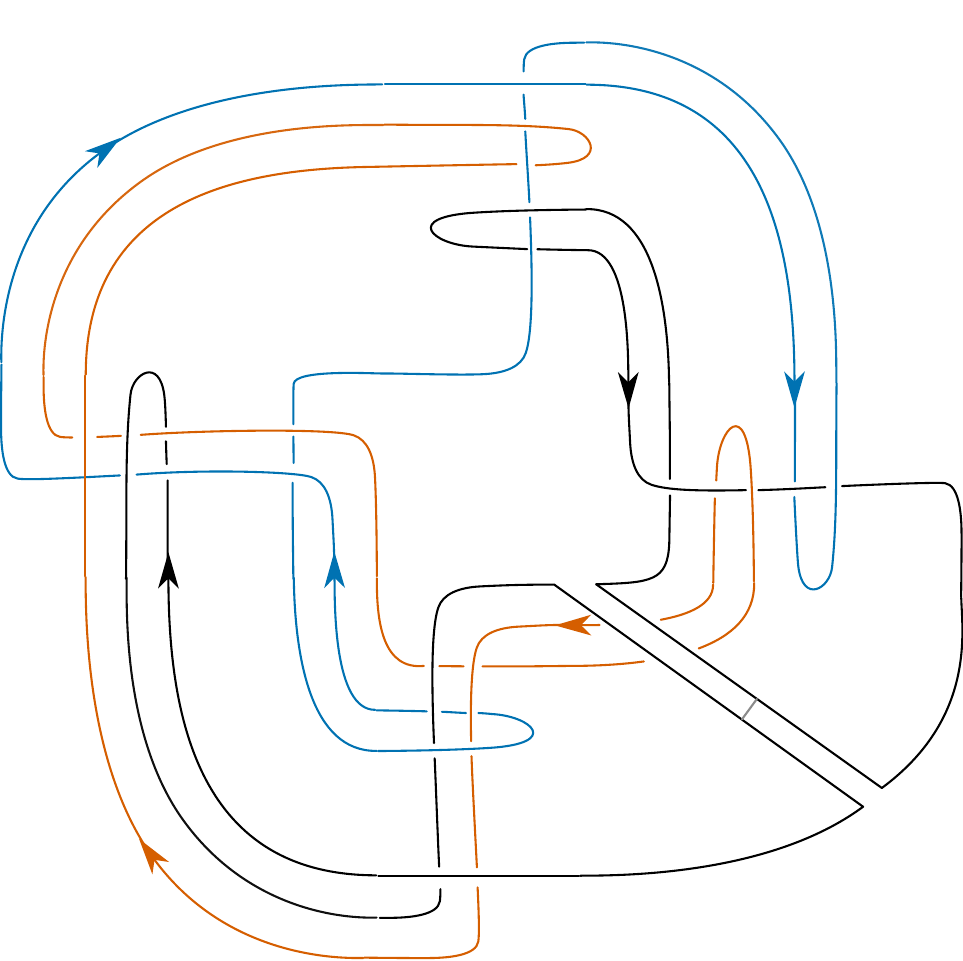\qquad\qquad
\def\svgwidth{0.4\textwidth}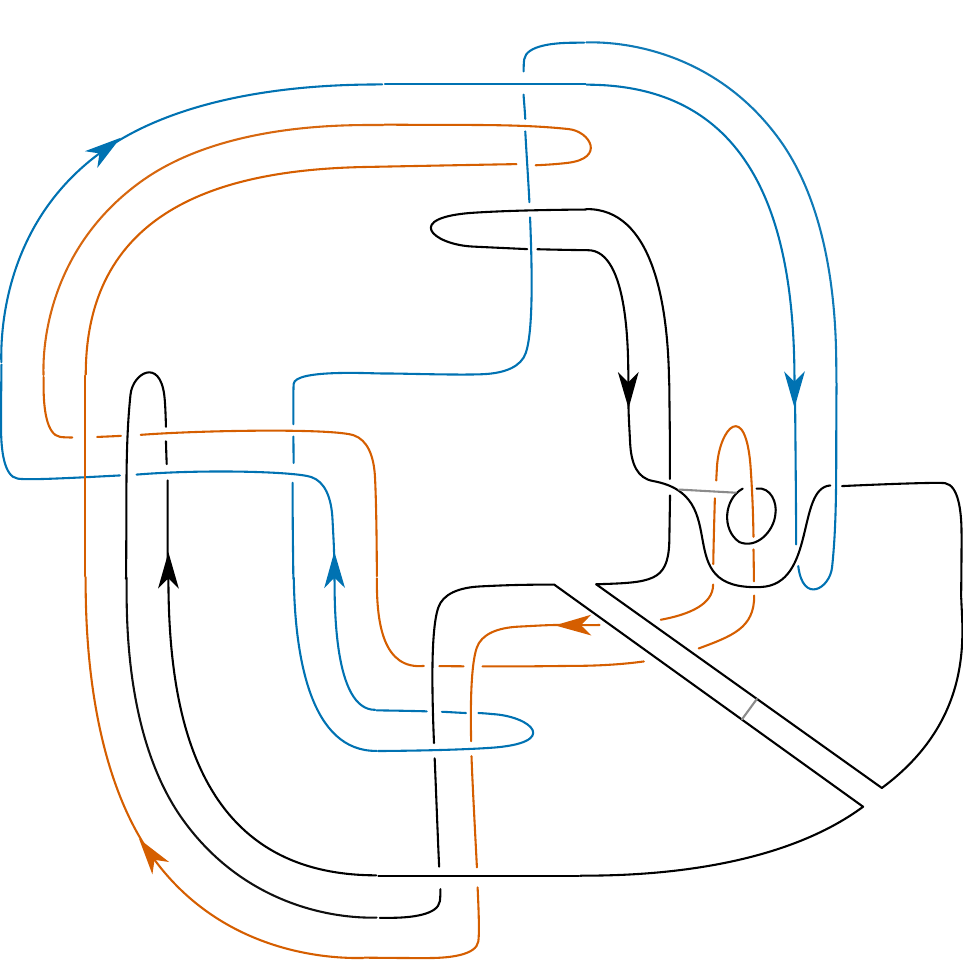
\caption{The sublink $\gamma_0\cup \gamma_1\cup \gamma_2$. The bands appear as grey intervals.}\label{L023}
\end{figure}

\begin{figure}
\centering
\centering
\def\svgwidth{0.375\textwidth}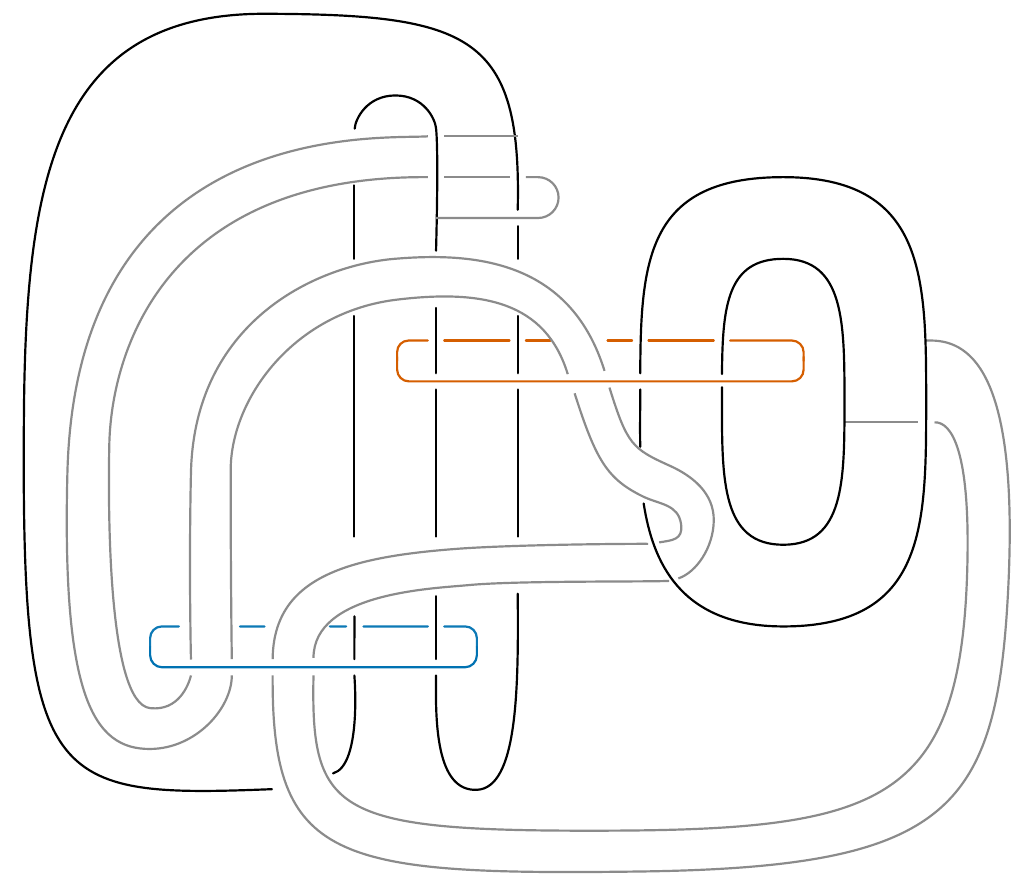 \qquad
\def\svgwidth{0.15\textwidth}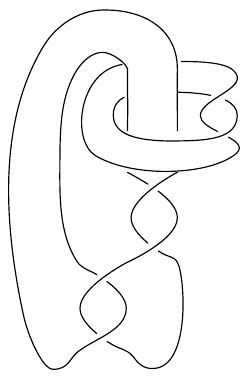 \qquad
\def\svgwidth{0.325\textwidth}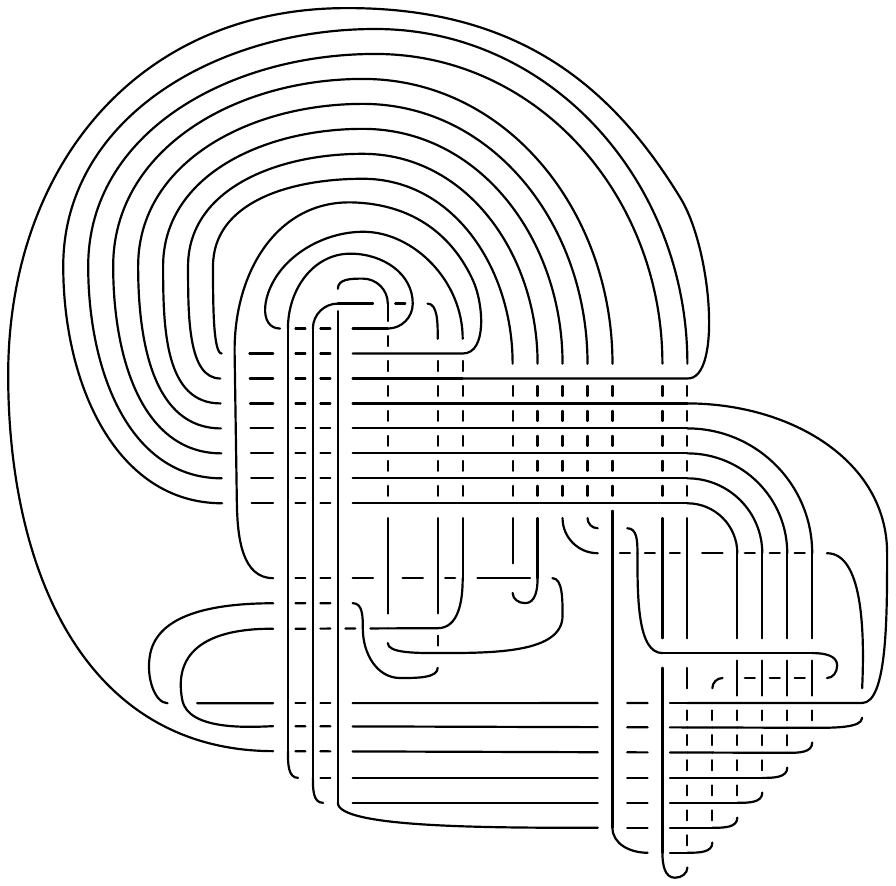

\caption{Left: The sublink $\gamma_0\cup \gamma_1\cup \gamma_2$, with $\gamma_0$ represented as a band sum. The numbers next to the bands represent the number of half-twists. Both curves $\gamma_2$ and $\gamma_3$ have framing $-1$. Middle: The effect of blowing down $\gamma_1\cup \gamma_2$ on $\alpha_0\cup \alpha_1\cup \alpha_2$. This is a 3-component link  Right: The knot $J_0$ obtained from $\gamma_0$ after blowing down $\gamma_2\cup\gamma_3$.}\label{J0}
\end{figure}

We're now in position to prove \Cref{order-cover}, asserting that $\Sigma_4(C)$ has infinite order in the rational homology cobordism group.

\begin{proof}[Proof of \Cref{order-cover}]
By \Cref{p:two-component}, $\Sigma=\Sigma_4(C)$ is obtained by doing $-1$-surgery along each of the components of the link $J_0 \cup J_1$. Let $W$ be the cobordism from $S^3$ to $\Sigma$ obtained from $[0,1]\times S^3$ by attaching two 2-handles, along $J_0$ and $J_1$, and each with framing $-1$. Denote  $\Sigma' = S^3_{-1}(J_0)$, and decompose $W$ as $X_{-1}(J_0)\cup_{\Sigma'} W'$, where $X_{-1}(J_0)$ is the trace cobordism from $S^3$ to $\Sigma'$, and $W'$ the result of attaching the remaining 2-handle to a collar of $\Sigma'$, along the knot $J_1$ (regarded now as a knot in $\Sigma'$.) Since the linking number of $J_0$ and $J_1$ is $0$, the framing number of $J_0$ in $\Sigma'$ is still $-1$ and so $W'$ is a negative definite manifold (as is $W$).  

We can now readily obtain bounds for $d(\Sigma')$ and $d(\Sigma)$. Indeed, by \Cref{p:tau-negative} $\tau(J_0) < 0$, and so from \Cref{t:homwu} we conclude $d(\Sigma') > 0$. Next, the third item from \Cref{t:dinv} implies
\[
d := d(\Sigma) \ge d(\Sigma') > 0.
\]
Then, for any integer $n \neq 0$ we have $d(\#^n\Sigma_4(C)) = nd \neq 0$ and so $\#^n\Sigma_4(C)$ does not bound a rational homology ball by \Cref{t:dinv}. Therefore, $\Sigma_4(C)$ has infinite order in $\Theta^3_\Q$.
\end{proof}

\begin{remark}\label{r:summand}
Correction terms of branched covers have been used in ~\cite[Theorem~1.1]{ManolescuOwens} and \cite{Jabuka} to define concordance invariants.
In particular, since branched covers of the form $\Sigma_{2^k}(K)$ are $\Z/2\Z$-homology 3-spheres, they admit a unique spin structure $\mathfrak t_{\rm spin}$, and the map $\delta_{2^k} \colon \Con \to \Q$ defined as
\[
\delta_{2^k}(K):= 2d(\Sigma_{2^k}(K),\mathfrak{t}_{\rm spin})
\]
is an integer-valued homomorphism.
A common method for exhibiting summands of the concordance group $\Con$ is to produce an integer valued homomorphism that evaluates to $1$. For the case at hand, no homomorphism $\delta_{2^k} $ (or any $\delta_q$ for a prime power $q$) can be used to prove that $C$ generates a summand. This is because since $\Delta_C(t) = 1$, all its cyclic covers are integer homology spheres, and therefore $\delta_q(C) \in 4\Z$ for each $q$. Nevertheless, the fact that $d(\Sigma_4(C)) \neq 0$ also implies that $C$ is not divisible by arbitrarily large integers in $\Con$. Indeed, the fact that $\delta_4(C) = 2d \ge 4$ (and $d$ divisible by $2$) shows that the divisors of the class of $C$ in $\Con$ are necessarily divisors of $d$. 

As a final point, we remark that it is currently unknown if there is arbitrary torsion in $\Con$, or if there are (non-trivial) elements that are divisible by arbitrarily large integers.
This prevents us from being able to prove that $C$ even \emph{belongs} to a $\Z$-summand in $\Con$, even up to torsion.
\end{remark}

\begin{remark}
One can give an estimate on $d(\Sigma_4(C))$ in terms of slice surfaces bounded by $J_0 \cup J_1$.
For instance, one can show that $g_4(J_0) = 1$, which, by~\cite{Rasmussen-GodaTeragaito, NiWu}, implies that $d(S^3_{-1}(J_0))= 2$.
If one could show that $J_1$ is slice in $S^3_{-1}(J_0)$ (in the sense that it bounds a smoothly embedded disc in $[0,1]\times S^3_{-1}(J_0)$), then it would follow that $d(\Sigma_4(C)) = 2$, and so $\delta_4(C)=4$.
Even in this case, the best conclusion would be the existence of a $\Z$-summand in $\Con_\Delta$, the subgroup of $\Con$ consisting of knots with trivial Alexander polynomial.
\end{remark}

\section{The generalised Conway knots $C_{2,n}$}
In this section, we prove \Cref{C2n}, asserting that the Conway knots $C_{2,n}$, too, have infinite order in $\Con$.

\begin{figure}
\centering
\def\svgwidth{0.275\textwidth}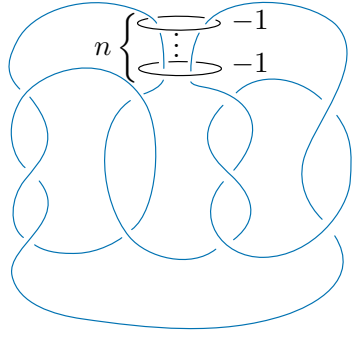 \qquad
\caption{A surgery description of the knot $C_{2,n}$ that gives rise to a negative-definite cobordism from $C'=C_{2,0}$ to $C_{2,n}$.}\label{f:cobordismC2n-bis}
\end{figure}

\begin{proof}[Proof of \Cref{C2n}]
Similar to \Cref{order-C}, the result will be established after showing that the $d$-invariant of $\Sigma_n=\Sigma_4(C_{2,n})$ is non-trivial for every $n\neq 0$. Since $m\left(C_{2,n}\right)=C_{2,-n}$, it is then enough to show that $d(\Sigma_4(C_{2,n})) >0$ for $n>0$. We will do so by exhibiting a negative definite cobordism $X_n$ from $\Sigma=\Sigma_{1}=\Sigma_4(C)$ to $\Sigma_n$. 

To describe $X_n$, consider the cobordism from $C'=C_{2,0}$ to $C_{2,n}$, shown in \Cref{f:cobordismC2n-bis}. 
That is, consider the blow-up $P_n = [0,1] \times S^3 \# n\overline{\mathbb{CP}}\vphantom{\C}^2$, and notice that there is a null-homologous genus-0 cobordism $F_n\subset P_n$ from $C'$ to $C_{2,n}$.
Let $W_n$ be the 4-fold cover of $P_n$ branched over $F_n$, and notice that $W_n$ is the result of attaching $4n$ 2-handles along the lifts of $n$ $0$-framed push-offs of the curve $\gamma$ to $S^3=\Sigma_4(C')$.
In the proof of \Cref{p:two-component}, we showed that (1) $\gamma$ lifts to $\gamma_0\cup \gamma_1\cup \gamma_2\cup \gamma_3$, (2) the Seifert framing on $\gamma$ lifts to the Seifert framings on each $\gamma_i$, and (3) $\lk(\gamma_i, \gamma_j) = 0$ (see \Cref{4-fold}).
It thus follows that $W_n$ has intersection form isomorphic to $\langle -1\rangle^{\oplus 4n}$, and so $W_n$ is negative definite for any $n\geq 1$. Finally,  notice that $W_n$ factors as $W_n=W_1\underset{\Sigma}{\cup} X_n$ with $X_n:= \left(W_n\setminus W_1\right)$. That is, $X_n$ is the result of factoring $W_n$ through $\Sigma$.
Since $\Sigma$ is an integer homology sphere, $X_n$ is a negative-definite cobordism from $\Sigma=\Sigma_4(C)$ to $\Sigma_n$.

This, together with \Cref{t:dinv}, shows that $d(\Sigma_4(C_{2,n})) \ge d(\Sigma_4(C))$. Since the latter was proven to be strictly positive in the proof of \Cref{order-cover}, we have that $\Sigma_n=\Sigma_4(C_{2,n})$ has infinite order in $\Theta^3_\Q$, and therefore $C_{2,n}$ has infinite order in $\Con$.
\end{proof}

\vfill

\pagebreak
\appendix

\section{Isotopies}
This appendix shows the explicit isotopies used in \Cref{p:two-component}, and in the description of $J_0$ included in \Cref{tau-bounds}.  
\begin{figure}[H]
\centering
\includegraphics[width=0.8125\textwidth]{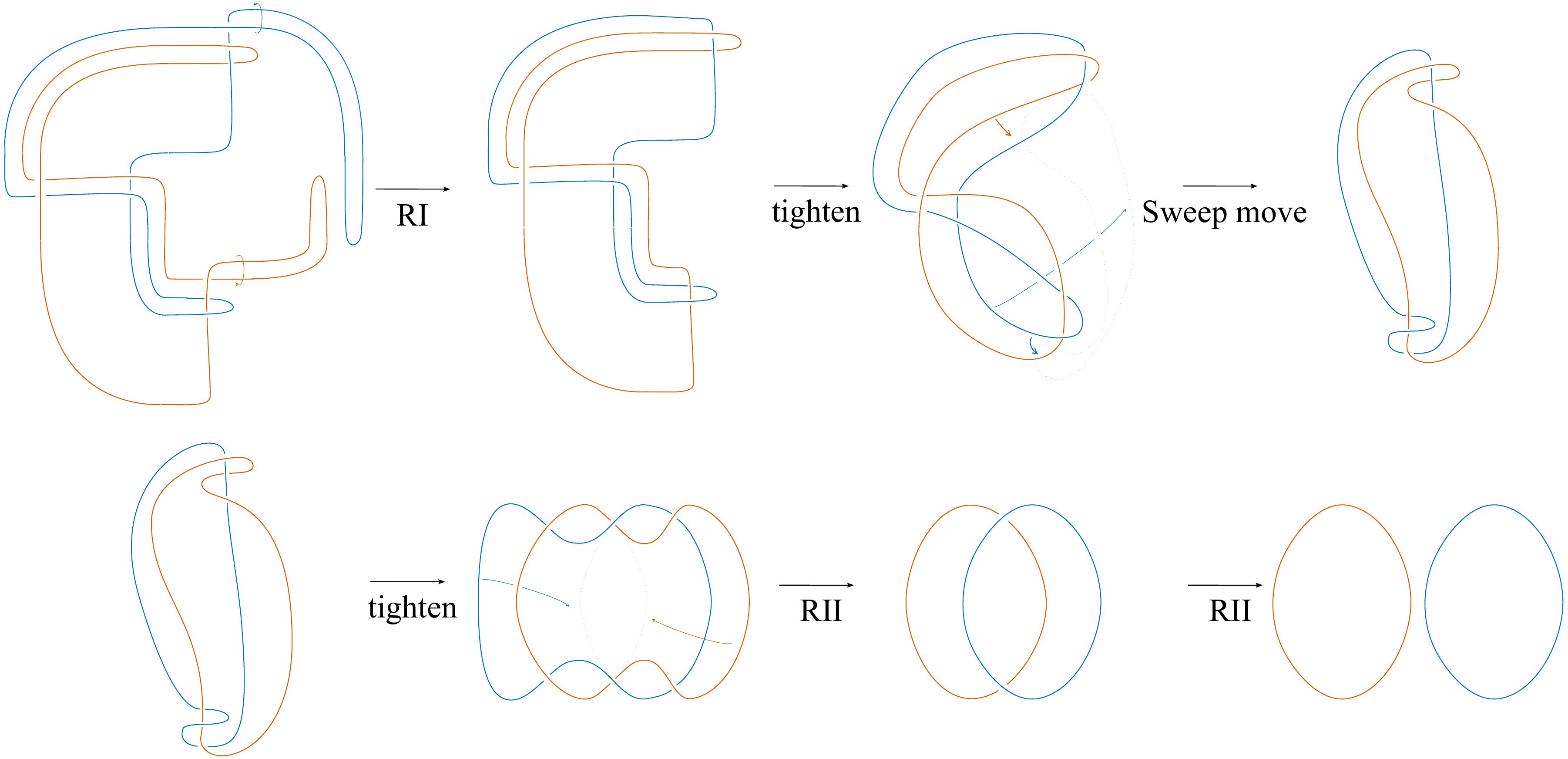}
\caption{Isotopies showing the equivalence between $\gamma_i\cup \gamma_j$ and the 2-component unlink for $|i-j|=1 \pmod 4 $. }\label{consecutive}
\end{figure}

\begin{figure}[H]
\centering
\includegraphics[width=0.8125\textwidth]{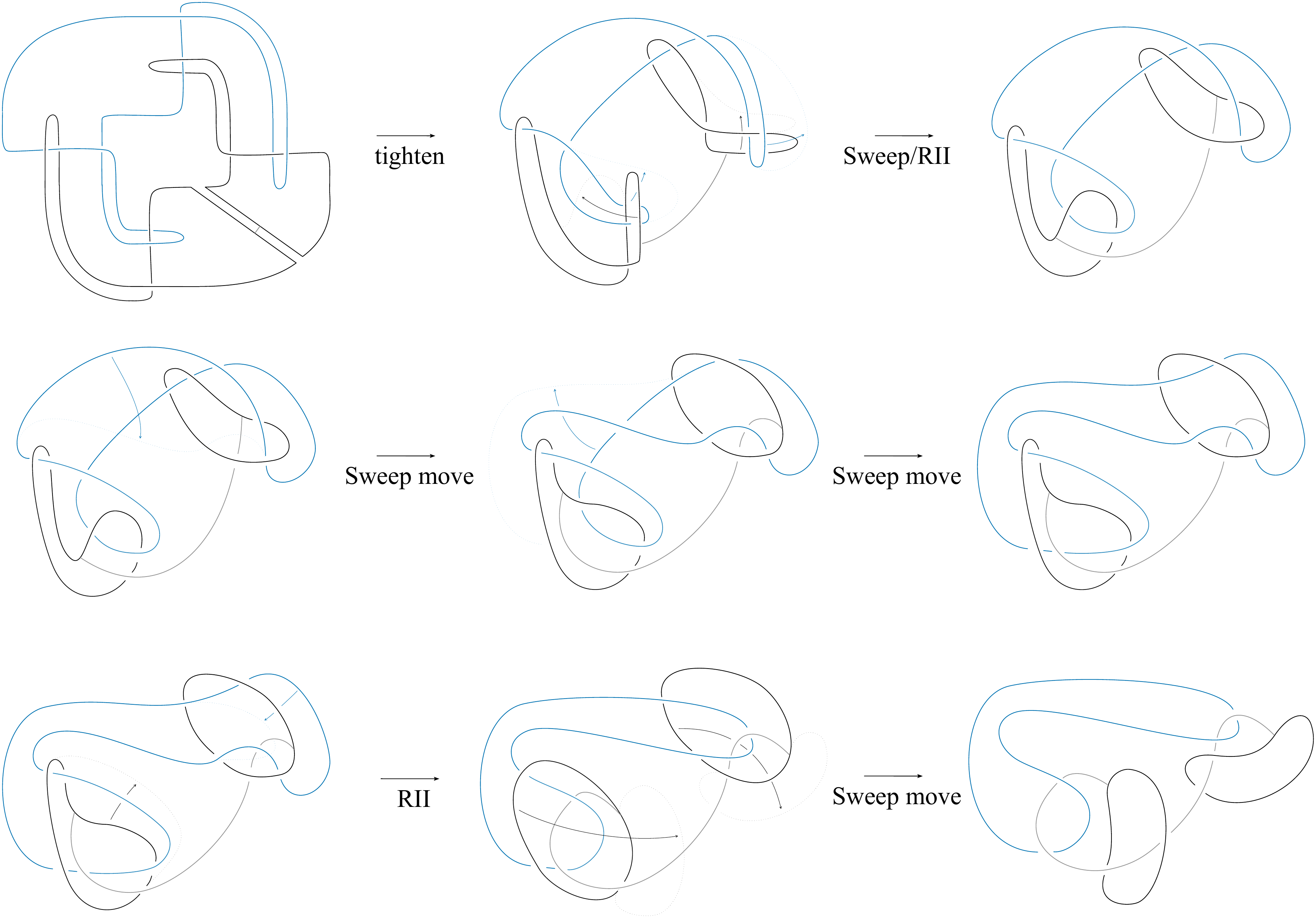}
\caption{Diagrams of the sublink $\gamma_i\cup\gamma_j$ with $|i-j|=2\pmod 4$. In the first diagram the curve $\gamma_0$ is realized as the banded sum of two unknots. The last diagram shows that these links are slice.}\label{non-consecutive}
\end{figure}

\bibliographystyle{alphaurl}
\bibliography{conway}

\end{document}

%% file: Conway.pdf_tex
\begingroup%
  \makeatletter%
  \providecommand\color[2][]{%
    \errmessage{(Inkscape) Color is used for the text in Inkscape, but the package 'color.sty' is not loaded}%
    \renewcommand\color[2][]{}%
  }%
  \providecommand\transparent[1]{%
    \errmessage{(Inkscape) Transparency is used (non-zero) for the text in Inkscape, but the package 'transparent.sty' is not loaded}%
    \renewcommand\transparent[1]{}%
  }%
  \providecommand\rotatebox[2]{#2}%
  \newcommand*\fsize{\dimexpr\f@size pt\relax}%
  \newcommand*\lineheight[1]{\fontsize{\fsize}{#1\fsize}\selectfont}%
  \ifx\svgwidth\undefined%
    \setlength{\unitlength}{293.09566227bp}%
    \ifx\svgscale\undefined%
      \relax%
    \else%
      \setlength{\unitlength}{\unitlength * \real{\svgscale}}%
    \fi%
  \else%
    \setlength{\unitlength}{\svgwidth}%
  \fi%
  \global\let\svgwidth\undefined%
  \global\let\svgscale\undefined%
  \makeatother%
  \begin{picture}(1,0.97259944)%
    \lineheight{1}%
    \setlength\tabcolsep{0pt}%
    \put(0,0){\includegraphics[width=\unitlength,page=1]{Conway.pdf}}%
    \put(0.45,0.05){\color{azul}{$C$}}
  \end{picture}%
\endgroup%

%% file: C-unknotting.pdf_tex
\begingroup%
  \makeatletter%
  \providecommand\color[2][]{%
    \errmessage{(Inkscape) Color is used for the text in Inkscape, but the package 'color.sty' is not loaded}%
    \renewcommand\color[2][]{}%
  }%
  \providecommand\transparent[1]{%
    \errmessage{(Inkscape) Transparency is used (non-zero) for the text in Inkscape, but the package 'transparent.sty' is not loaded}%
    \renewcommand\transparent[1]{}%
  }%
  \providecommand\rotatebox[2]{#2}%
  \newcommand*\fsize{\dimexpr\f@size pt\relax}%
  \newcommand*\lineheight[1]{\fontsize{\fsize}{#1\fsize}\selectfont}%
  \ifx\svgwidth\undefined%
    \setlength{\unitlength}{293.09566227bp}%
    \ifx\svgscale\undefined%
      \relax%
    \else%
      \setlength{\unitlength}{\unitlength * \real{\svgscale}}%
    \fi%
  \else%
    \setlength{\unitlength}{\svgwidth}%
  \fi%
  \global\let\svgwidth\undefined%
  \global\let\svgscale\undefined%
  \makeatother%
  \begin{picture}(1,1)%
    \lineheight{1}%
    \setlength\tabcolsep{0pt}%
    \put(0,0){\includegraphics[width=\unitlength,page=1]{C-unknotting.pdf}}%
    \put(0.3,0.8){$\gamma$}%
    \put(0.65,0.775){$-1$}%
    \put(0.45,0.05){\color{azul}$C'$}
  \end{picture}%
\endgroup%

%% file: C-unknotted.pdf_tex
\begingroup%
  \makeatletter%
  \providecommand\color[2][]{%
    \errmessage{(Inkscape) Color is used for the text in Inkscape, but the package 'color.sty' is not loaded}%
    \renewcommand\color[2][]{}%
  }%
  \providecommand\transparent[1]{%
    \errmessage{(Inkscape) Transparency is used (non-zero) for the text in Inkscape, but the package 'transparent.sty' is not loaded}%
    \renewcommand\transparent[1]{}%
  }%
  \providecommand\rotatebox[2]{#2}%
  \newcommand*\fsize{\dimexpr\f@size pt\relax}%
  \newcommand*\lineheight[1]{\fontsize{\fsize}{#1\fsize}\selectfont}%
  \ifx\svgwidth\undefined%
    \setlength{\unitlength}{298.18882019bp}%
    \ifx\svgscale\undefined%
      \relax%
    \else%
      \setlength{\unitlength}{\unitlength * \real{\svgscale}}%
    \fi%
  \else%
    \setlength{\unitlength}{\svgwidth}%
  \fi%
  \global\let\svgwidth\undefined%
  \global\let\svgscale\undefined%
  \makeatother%
  \begin{picture}(1,1)%
    \lineheight{1}%
    \setlength\tabcolsep{0pt}%
    \put(0,0){\includegraphics[width=\unitlength,page=1]{C-unknotted.pdf}}%
    \put(1,0.35){$\gamma$}%
    \put(0.85,0.85){$-1$}%
    \put(0.95,0.6){\color{azul}$C'$}
  \end{picture}%
\endgroup%

%% file: C_r,n.pdf_tex
\begingroup%
  \makeatletter%
  \providecommand\color[2][]{%
    \errmessage{(Inkscape) Color is used for the text in Inkscape, but the package 'color.sty' is not loaded}%
    \renewcommand\color[2][]{}%
  }%
  \providecommand\transparent[1]{%
    \errmessage{(Inkscape) Transparency is used (non-zero) for the text in Inkscape, but the package 'transparent.sty' is not loaded}%
    \renewcommand\transparent[1]{}%
  }%
  \providecommand\rotatebox[2]{#2}%
  \newcommand*\fsize{\dimexpr\f@size pt\relax}%
  \newcommand*\lineheight[1]{\fontsize{\fsize}{#1\fsize}\selectfont}%
  \ifx\svgwidth\undefined%
    \setlength{\unitlength}{311.81102362bp}%
    \ifx\svgscale\undefined%
      \relax%
    \else%
      \setlength{\unitlength}{\unitlength * \real{\svgscale}}%
    \fi%
  \else%
    \setlength{\unitlength}{\svgwidth}%
  \fi%
  \global\let\svgwidth\undefined%
  \global\let\svgscale\undefined%
  \makeatother%
  \begin{picture}(1,0.7)%
    \lineheight{1}%
    \setlength\tabcolsep{0pt}%
    \put(0,0){\includegraphics[width=\unitlength,page=1]{C_r,n.pdf}}%
    \put(0.455,0.525){$2n$}
    \put(0.0175,0.325){\tiny$r+1$}
    \put(0.3125,0.325){$-r$}
    \put(0.55,0.325){\tiny$-r-1$}
    \put(0.875,0.325){$r$}
  \end{picture}%
\endgroup%

%% file: KT_r,n.pdf_tex
\begingroup%
  \makeatletter%
  \providecommand\color[2][]{%
    \errmessage{(Inkscape) Color is used for the text in Inkscape, but the package 'color.sty' is not loaded}%
    \renewcommand\color[2][]{}%
  }%
  \providecommand\transparent[1]{%
    \errmessage{(Inkscape) Transparency is used (non-zero) for the text in Inkscape, but the package 'transparent.sty' is not loaded}%
    \renewcommand\transparent[1]{}%
  }%
  \providecommand\rotatebox[2]{#2}%
  \newcommand*\fsize{\dimexpr\f@size pt\relax}%
  \newcommand*\lineheight[1]{\fontsize{\fsize}{#1\fsize}\selectfont}%
  \ifx\svgwidth\undefined%
    \setlength{\unitlength}{311.81102362bp}%
    \ifx\svgscale\undefined%
      \relax%
    \else%
      \setlength{\unitlength}{\unitlength * \real{\svgscale}}%
    \fi%
  \else%
    \setlength{\unitlength}{\svgwidth}%
  \fi%
  \global\let\svgwidth\undefined%
  \global\let\svgscale\undefined%
  \makeatother%
  \begin{picture}(1,0.68181818)%
    \lineheight{1}%
    \setlength\tabcolsep{0pt}%
    \put(0,0){\includegraphics[width=\unitlength,page=1]{C_r,n.pdf}}%
    \put(0.455,0.525){$2n$}
    \put(0.0175,0.325){\tiny$r+1$}
    \put(0.3125,0.325){$-r$}
    \put(0.6,0.325){$r$}
    \put(0.8,0.325){\tiny$-r-1$}
  \end{picture}%
\endgroup%

%% file: 4fold.pdf_tex
\begingroup%
  \makeatletter%
  \providecommand\color[2][]{%
    \errmessage{(Inkscape) Color is used for the text in Inkscape, but the package 'color.sty' is not loaded}%
    \renewcommand\color[2][]{}%
  }%
  \providecommand\transparent[1]{%
    \errmessage{(Inkscape) Transparency is used (non-zero) for the text in Inkscape, but the package 'transparent.sty' is not loaded}%
    \renewcommand\transparent[1]{}%
  }%
  \providecommand\rotatebox[2]{#2}%
  \newcommand*\fsize{\dimexpr\f@size pt\relax}%
  \newcommand*\lineheight[1]{\fontsize{\fsize}{#1\fsize}\selectfont}%
  \ifx\svgwidth\undefined%
    \setlength{\unitlength}{462.0499258bp}%
    \ifx\svgscale\undefined%
      \relax%
    \else%
      \setlength{\unitlength}{\unitlength * \real{\svgscale}}%
    \fi%
  \else%
    \setlength{\unitlength}{\svgwidth}%
  \fi%
  \global\let\svgwidth\undefined%
  \global\let\svgscale\undefined%
  \makeatother%
  \begin{picture}(1,0.99788073)%
    \lineheight{1}%
    \setlength\tabcolsep{0pt}%
    \put(0,0){\includegraphics[width=\unitlength,page=1]{4fold.pdf}}%
            \put(0.5,0.425){$\gamma_0$}
    \put(0.4,0.47){\color{rojo}$\gamma_3$}
     \put(0.45,0.57){\color{azul}$\gamma_2$}
    \put(0.545,0.51){\color{verde}$\gamma_1$}
    \put(0.025,0.2){\color{rojo}$-1$}
    \put(0.9,0.8){\color{verde}$-1$}
    \put(0.2,0.925){\color{azul}$-1$}
     \put(0.8,0.075){$-1$}
  \end{picture}%
\endgroup%

%% file: L023-A.pdf_tex
\begingroup%
  \makeatletter%
  \providecommand\color[2][]{%
    \errmessage{(Inkscape) Color is used for the text in Inkscape, but the package 'color.sty' is not loaded}%
    \renewcommand\color[2][]{}%
  }%
  \providecommand\transparent[1]{%
    \errmessage{(Inkscape) Transparency is used (non-zero) for the text in Inkscape, but the package 'transparent.sty' is not loaded}%
    \renewcommand\transparent[1]{}%
  }%
  \providecommand\rotatebox[2]{#2}%
  \newcommand*\fsize{\dimexpr\f@size pt\relax}%
  \newcommand*\lineheight[1]{\fontsize{\fsize}{#1\fsize}\selectfont}%
  \ifx\svgwidth\undefined%
    \setlength{\unitlength}{462.0499258bp}%
    \ifx\svgscale\undefined%
      \relax%
    \else%
      \setlength{\unitlength}{\unitlength * \real{\svgscale}}%
    \fi%
  \else%
    \setlength{\unitlength}{\svgwidth}%
  \fi%
  \global\let\svgwidth\undefined%
  \global\let\svgscale\undefined%
  \makeatother%
  \begin{picture}(1,1)%
    \lineheight{1}%
    \setlength\tabcolsep{0pt}%
    \put(0,0){\includegraphics[width=\unitlength,page=1]{L023-A.pdf}}%
     \put(0.4,0.47){\color{rojo}$\gamma_3$}
     \put(0.45,0.57){\color{azul}$\gamma_2$}
    \put(0.9,0.8){$-1$}
    \put(0.2,0.925){$-1$}
     \put(0.7,0.125){$\alpha'_1$}
     \put(0.9,0.275){$\alpha'_2$}
  \end{picture}%
\endgroup%

%% file: L023-B.pdf_tex
\begingroup%
  \makeatletter%
  \providecommand\color[2][]{%
    \errmessage{(Inkscape) Color is used for the text in Inkscape, but the package 'color.sty' is not loaded}%
    \renewcommand\color[2][]{}%
  }%
  \providecommand\transparent[1]{%
    \errmessage{(Inkscape) Transparency is used (non-zero) for the text in Inkscape, but the package 'transparent.sty' is not loaded}%
    \renewcommand\transparent[1]{}%
  }%
  \providecommand\rotatebox[2]{#2}%
  \newcommand*\fsize{\dimexpr\f@size pt\relax}%
  \newcommand*\lineheight[1]{\fontsize{\fsize}{#1\fsize}\selectfont}%
  \ifx\svgwidth\undefined%
    \setlength{\unitlength}{462.0499258bp}%
    \ifx\svgscale\undefined%
      \relax%
    \else%
      \setlength{\unitlength}{\unitlength * \real{\svgscale}}%
    \fi%
  \else%
    \setlength{\unitlength}{\svgwidth}%
  \fi%
  \global\let\svgwidth\undefined%
  \global\let\svgscale\undefined%
  \makeatother%
  \begin{picture}(1,1)%
    \lineheight{1}%
    \setlength\tabcolsep{0pt}%
    \put(0,0){\includegraphics[width=\unitlength,page=1]{L023-B.pdf}}%
  \put(0.4,0.47){\color{rojo}$\gamma_3$}
     \put(0.45,0.57){\color{azul}$\gamma_2$}
    \put(0.9,0.8){$-1$}
    \put(0.2,0.925){$-1$}
     \put(0.75,0.55){$\alpha_2$}
     \put(0.9,0.275){$\alpha_0$}
     \put(0.7,0.125){$\alpha_1$}
  \end{picture}%
\endgroup%

%% file: J0.pdf_tex
\begingroup%
  \makeatletter%
  \providecommand\color[2][]{%
    \errmessage{(Inkscape) Color is used for the text in Inkscape, but the package 'color.sty' is not loaded}%
    \renewcommand\color[2][]{}%
  }%
  \providecommand\transparent[1]{%
    \errmessage{(Inkscape) Transparency is used (non-zero) for the text in Inkscape, but the package 'transparent.sty' is not loaded}%
    \renewcommand\transparent[1]{}%
  }%
  \providecommand\rotatebox[2]{#2}%
  \newcommand*\fsize{\dimexpr\f@size pt\relax}%
  \newcommand*\lineheight[1]{\fontsize{\fsize}{#1\fsize}\selectfont}%
  \ifx\svgwidth\undefined%
    \setlength{\unitlength}{496.06299213bp}%
    \ifx\svgscale\undefined%
      \relax%
    \else%
      \setlength{\unitlength}{\unitlength * \real{\svgscale}}%
    \fi%
  \else%
    \setlength{\unitlength}{\svgwidth}%
  \fi%
  \global\let\svgwidth\undefined%
  \global\let\svgscale\undefined%
  \makeatother%
  \begin{picture}(1,1)%
    \lineheight{1}%
    \setlength\tabcolsep{0pt}%
    \put(0,0){\includegraphics[width=\unitlength,page=1]{J0.pdf}}%
    \put(0.41,0.175){\color{azul}{$\gamma_2$}}
    \put(0.7,0.45){\color{rojo}$\gamma_3$}
    \put(0.1,0.85){$\alpha_0$}
    \put(0.7,0.7){$\alpha_1$}
    \put(0.7,0.61){$\alpha_2$}
    \put(0.95,0.25){\color{gray} $+1$}
    \put(0.85,0.25){\color{gray} $0$}
  \end{picture}%
\endgroup%

%% file: J0-link.pdf_tex
\begingroup%
  \makeatletter%
  \providecommand\color[2][]{%
    \errmessage{(Inkscape) Color is used for the text in Inkscape, but the package 'color.sty' is not loaded}%
    \renewcommand\color[2][]{}%
  }%
  \providecommand\transparent[1]{%
    \errmessage{(Inkscape) Transparency is used (non-zero) for the text in Inkscape, but the package 'transparent.sty' is not loaded}%
    \renewcommand\transparent[1]{}%
  }%
  \providecommand\rotatebox[2]{#2}%
  \newcommand*\fsize{\dimexpr\f@size pt\relax}%
  \newcommand*\lineheight[1]{\fontsize{\fsize}{#1\fsize}\selectfont}%
  \ifx\svgwidth\undefined%
    \setlength{\unitlength}{197.53588819bp}%
    \ifx\svgscale\undefined%
      \relax%
    \else%
      \setlength{\unitlength}{\unitlength * \real{\svgscale}}%
    \fi%
  \else%
    \setlength{\unitlength}{\svgwidth}%
  \fi%
  \global\let\svgwidth\undefined%
  \global\let\svgscale\undefined%
  \makeatother%
  \begin{picture}(1,1.55419278)%
    \lineheight{1}%
    \setlength\tabcolsep{0pt}%
    \put(0,0){\includegraphics[width=\unitlength,page=1]{J0-link.pdf}}%
  \end{picture}%
\endgroup%

%% file: J0-snappy.pdf_tex
\begingroup%
  \makeatletter%
  \providecommand\color[2][]{%
    \errmessage{(Inkscape) Color is used for the text in Inkscape, but the package 'color.sty' is not loaded}%
    \renewcommand\color[2][]{}%
  }%
  \providecommand\transparent[1]{%
    \errmessage{(Inkscape) Transparency is used (non-zero) for the text in Inkscape, but the package 'transparent.sty' is not loaded}%
    \renewcommand\transparent[1]{}%
  }%
  \providecommand\rotatebox[2]{#2}%
  \newcommand*\fsize{\dimexpr\f@size pt\relax}%
  \newcommand*\lineheight[1]{\fontsize{\fsize}{#1\fsize}\selectfont}%
  \ifx\svgwidth\undefined%
    \setlength{\unitlength}{429.73796694bp}%
    \ifx\svgscale\undefined%
      \relax%
    \else%
      \setlength{\unitlength}{\unitlength * \real{\svgscale}}%
    \fi%
  \else%
    \setlength{\unitlength}{\svgwidth}%
  \fi%
  \global\let\svgwidth\undefined%
  \global\let\svgscale\undefined%
  \makeatother%
  \begin{picture}(1,0.98943282)%
    \lineheight{1}%
    \setlength\tabcolsep{0pt}%
    \put(0,0){\includegraphics[width=\unitlength,page=1]{J0-snappy.pdf}}%
  \end{picture}%
\endgroup%

%% file: C_r,n-unknotting.pdf_tex
\begingroup%
  \makeatletter%
  \providecommand\color[2][]{%
    \errmessage{(Inkscape) Color is used for the text in Inkscape, but the package 'color.sty' is not loaded}%
    \renewcommand\color[2][]{}%
  }%
  \providecommand\transparent[1]{%
    \errmessage{(Inkscape) Transparency is used (non-zero) for the text in Inkscape, but the package 'transparent.sty' is not loaded}%
    \renewcommand\transparent[1]{}%
  }%
  \providecommand\rotatebox[2]{#2}%
  \newcommand*\fsize{\dimexpr\f@size pt\relax}%
  \newcommand*\lineheight[1]{\fontsize{\fsize}{#1\fsize}\selectfont}%
  \ifx\svgwidth\undefined%
    \setlength{\unitlength}{293.09566227bp}%
    \ifx\svgscale\undefined%
      \relax%
    \else%
      \setlength{\unitlength}{\unitlength * \real{\svgscale}}%
    \fi%
  \else%
    \setlength{\unitlength}{\svgwidth}%
  \fi%
  \global\let\svgwidth\undefined%
  \global\let\svgscale\undefined%
  \makeatother%
  \begin{picture}(1,1)%
    \lineheight{1}%
    \setlength\tabcolsep{0pt}%
    \put(0,0){\includegraphics[width=\unitlength,page=1]{C_r,n-unknotting.pdf}}%
     \put(0.65,0.75){$-1$}
     \put(0.65,0.875){$-1$}
     \put(0.4775,0.7925){$\vdots$}
     \put(0.25,0.8){$n\left\{\begin{array}{l}\phantom{a}\\\phantom{a}\end{array}\right.$}
  \end{picture}%
\endgroup%